\documentclass{amsart}%
\usepackage{amsfonts}
\usepackage{amsmath}
\usepackage{amssymb}
\usepackage{graphicx}
\usepackage{hyperref}%
\providecommand{\U}[1]{\protect\rule{.1in}{.1in}}
\newtheorem{theorem}{Theorem}
\theoremstyle{plain}

\newtheorem{corollary}{Corollary}

\newtheorem{definition}{Definition}
\newtheorem{example}{Example}

\newtheorem{lemma}{Lemma}

\newtheorem{problem}{Problem}

\newtheorem{remark}{Remark}

\numberwithin{equation}{section}
\begin{document}
\title[ ]{Choquet operators associated to vector capacities }
\author{Sorin G. Gal}
\address{Department of Mathematics and Computer Science\\
University of Oradea\\
University\ Street No. 1, Oradea, 410087, Romania}
\email{galsorin23@gmail.com}
\author{Constantin P. Niculescu}
\address{Department of Mathematics, University of Craiova\\
Craiova 200585, Romania}
\email{constantin.p.niculescu@gmail.com}
\thanks{Appears in J. Math. Anal. Appl. \textbf{500} (2021), Issue 2, Article No.
125153. DOI: 10.1016/j.jmaa.2021.125153.}
\date{April 12, 2021}

\begin{abstract}
The integral representation of Choquet operators defined on a space $C(X)$ is
established by using the Choquet-Bochner integral of a real-valued function
with respect to a vector capacity.

\end{abstract}
\maketitle

\section{Introduction}

Choquet's theory of integrability (as described by Denneberg \cite{Denn} and
Wang and Klir \cite{WK}) leads to a new class of nonlinear operators called in
\cite{GN2020b} \emph{Choquet operators} because they are defined by a mix of
conditions representative for Choquet's integral. Its technical definition is
detailed as follows.

Given a Hausdorff topological space $X,$ we will denote by $\mathcal{F}(X)$
the vector lattice of all real-valued functions defined on $X$ endowed with
the pointwise ordering. Two important vector sublattices of it are
\[
C(X)=\left\{  f\in\mathcal{F}(X):\text{ }f\text{ continuous}\right\}
\]
and
\[
C_{b}(X)=\left\{  f\in\mathcal{F}(X):\text{ }f\text{ continuous and
bounded}\right\}  .
\]
With respect to the sup norm, $C_{b}(X)$ becomes a Banach lattice. See the
next section for details concerning the ordered Banch spaces.

As is well known, all norms on the $N$-dimensional real vector space
$\mathbb{R}^{N}$ are equivalent. See Bhatia \cite{Bhatia2009}, Theorem 13, p.
16. When endowed with the sup norm and the coordinate wise ordering,
$\mathbb{R}^{N}$ can be identified (algebraically, isometrically and in order)
with the space $C\left(  \left\{  1,...,N\right\}  \right)  $, where $\left\{
1,...,N\right\}  $ carries the discrete topology.

Suppose that $X$ and $Y$ are two Hausdorff topological spaces and $E$ and $F$
are respectively ordered vector subspaces of $\mathcal{F}(X)$ and
$\mathcal{F}(Y).$ An operator $T:E\rightarrow F$ is said to be a \emph{Choquet
operator }(respectively a\emph{ Choquet functional when }$F=\mathbb{R}$) if it
satisfies the following three conditions:

\begin{enumerate}
\item[(Ch1)] (\emph{Sublinearity}) $T$ is subadditive and positively
homogeneous, that is,%
\[
T(f+g)\leq T(f)+T(g)\quad\text{and}\quad T(af)=aT(f)
\]
for all $f,g$ in $E$ and $a\geq0;$

\item[(Ch2)] (\emph{Comonotonic additivity}) $T(f+g)=T(f)+T(g)$ whenever the
functions $f,g\in E$ are comonotone in the sense that
\[
(f(s)-f(t))\cdot(g(s)-g(t))\geq0\text{ for all }s,t\in X;
\]

\item[(Ch3)] (\emph{Monotonicity}) $f\leq g$ in $E$ implies $T(f)\leq T(g).$
\end{enumerate}

The linear Choquet operators acting on ordered Banach spaces are nothing but
the linear and positive operators acting on these spaces; see Corollary 1.
While they are omnipresent in the various fields of mathematics, the nonlinear
Choquet operators are less visible, their study beginning with the\ seminal
papers of Schmeidler \cite{Schm86}, \cite{Schm89} in the 80's. An important
step ahead was done by the contributions of Zhou \cite{Zhou}, Marinacci and
Montrucchio \cite{MM2004} and Cerreia-Vioglio, Maccheroni, Marinacci and
Montrucchio \cite{CMMM2012}, \cite{CMMM2015}, which led to the study of
vector-valued Choquet operators in their own. Notice that other important
contributions concerning this topic were obtained by Mesiar \cite{Mes_1},
Mesiar-Li-Pap \cite{MP}, Ouyang-Li-Mesiar \cite{Mes_2} and Pap \cite{Pap}. See
also Gal-Niculescu \cite{GN2020a} and \cite{GN2020b}.

Interestingly, the condition of comonotonic additivity (the substitute for
additivity) lies at the core of many results concerning the real analysis.
Indeed, its meaning in the context of real numbers, can be easily understood
by identifying each real number $x$ with the linear function $\alpha_{x}(t)=x
t$, $t\in\mathbb{R}.$ As a consequence, two real numbers $x$ and $y$ are
comonotone if and only if the functions $\alpha_{x}$ and $\alpha_{y}$ are
comonotone, equivalently, if either both numbers are nonnegative or both are
nonpositive. This yields the simplest example of Choquet functional from
$\mathbb{R}$ into itself which is not linear, the function $x\rightarrow
x^{+}.$ Let $C_{0}([-1, 1]) = \{f \in C([- 1, 1]) : f(0) = 0\}$. A large
family of nonlinear Choquet operators from the Banach lattice $C_{0}([-1, 1])$
to an arbitrary ordered Banach space $E$ is that of operators of the form
\[
T_{\varphi,U}(f)=U\left(  \int_{-1}^{1}f^{+}(tx)\varphi(x)\mathrm{d}x \right)
,
\]
where $\varphi\in C([-1,1])$ is any nonnegative function and $U:C_{0}%
([-1,1])\rightarrow E$ is any monotone linear operator.

Based on previous work done by Zhou \cite{Zhou}, Cerreia-Vioglio, Maccheroni,
Marinacci and Montrucchio \cite{CMMM2012}, proved that a larger class of
functionals defined on a space $C(X)$ (where $X$ is a Hausdorff compact space)
admit a Choquet analogue of the Riesz representation theorem. The aim of our
paper is to further extend their results to the case of operators by
developing a Choquet-Bochner theory of integration relative to monotone set
functions taking values in ordered Banach spaces.

Section 2 is devoted to a quick review of some basic facts from the theory of
ordered Banach spaces. While the particular case of Banach lattices is nicely
covered by a series of textbooks such as those by Meyer-Nieberg \cite{MN} and
Schaefer \cite{Sch1974}, the general theory of ordered Banach spaces is still
waiting to become the subject of an authoritative book.

In Section 3 we develop the theory of Choquet-Bochner integral associated to a
vector capacity (that is, to a monotone set function $\boldsymbol{\mu}$ taking
values in an ordered Banach space such that $\boldsymbol{\mu}(\emptyset)=0).$
As is shown in Theorem 1, this integral has all nice features of the Choquet
integral: monotonicity, positive homogeneity and comonotonic additivity. The
transfer of properties from vector capacities to their integrals also works in
a number of important cases such as the upper/lower continuity and
submodularity. See Theorem \ref{thmtransfer}. In the case of submodular vector
capacities with values in a Banach lattice, the integral analogue of the
modulus inequality also holds. See Theorem \ref{subadthm}.

Section 4 deals with the integral representation of the Choquet operators
defined on spaces $C(X)$ ($X$ being compact and Hausdorff) and taking values
in a Banach lattice with order continuous norm. The main result, Theorem
\ref{thmEWgen}, shows that each such operator is the Choquet-Bochner integral
associated to a suitable upper continuous vector capacity. In Section 5, this
representation is generalized to the framework of comonotonic additive
operators with bounded variation. See Theorem \ref{thmfinal}. The basic
ingredient is Lemma \ref{lemtech}, which shows that every comonotonic additive
operator with bounded variation can be written as the difference of two
positively homogeneous, translation invariant and monotone operators.

The paper ends with a short list of open problems.

\section{Preliminaries on ordered Banach spaces}

An \emph{ordered vector space} is a real vector space $E$ endowed with an
order relation $\leq$ such that the following two conditions are verified:%
\begin{align*}
x  &  \leq y\text{ implies }x+z\leq y+z\text{ for all }x,y,z\in E;\text{
and}\\
x  &  \leq y\text{ implies }\lambda x\leq\lambda y\text{ for }x,y\in E\text{
and }\lambda\in\mathbb{R}_{+}=[0,\infty).
\end{align*}

In this case the set $E_{+}=\left\{  x\in E:x\geq0\right\}  $ is a convex
cone, called the \emph{positive cone}. A real Banach space endowed with an
order relation that makes it an ordered vector space is called an
\emph{ordered Banach space} if the norm is monotone on the positive cone, that
is,%
\[
0\leq x\leq y\text{ implies }\left\Vert x\right\Vert \leq\left\Vert
y\right\Vert .
\]
Note that in this paper we will consider only ordered Banch spaces $E$ whose
positive cones are closed (in the norm topology), \emph{proper} $(-E_{+}\cap
E_{+}=\left\{  0\right\}  )$ and \emph{generating} $(E=E_{+}-E_{+})$.

A convenient way to emphasize the properties of ordered Banach spaces is that
described by Davies in \cite{Davies1968}. According to Davies, a real Banach
space $E$ endowed with a closed and generating cone $E_{+}$ such that%
\[
\left\Vert x\right\Vert =\inf\left\{  \left\Vert y\right\Vert :y\in E,\text{
}-y\leq x\leq y\right\}  \text{\quad for all }x\in E,
\]
is called a \emph{regularly ordered Banach space}%
\index{space!regularly ordered}%
.\emph{ }Examples are the Banach lattices and some other spaces such as
$\operatorname*{Sym}(n,\mathbb{R)}$, the ordered Banach space of all $n\times
n$-dimensional symmetric matrices with real coefficients. The norm of a
symmetric matrix $A$ is defined by the formula%
\[
\left\Vert A\right\Vert =\sup_{\left\Vert x\right\Vert \leq1}\left\vert
\langle Ax,x\rangle\right\vert ,
\]
and the positive cone $\operatorname*{Sym}^{+}(n,\mathbb{R})$ of
$\operatorname*{Sym}(n,\mathbb{R})$ consists of all symmetric matrices $A$
such that $\langle Ax,x\rangle\geq0$ for all $x.$

\begin{lemma}
\label{Lem0}Every ordered Banach space can be renormed by an equivalent norm
to become a regularly ordered Banach space.
\end{lemma}

For details, see Namioka \cite{Nam}. Some other useful properties of ordered
Banach spaces are listed below.

\begin{lemma}
\label{lem1}Suppose that $E$ is a regularly ordered Banach space. Then:

$(a)$ There exists a constant $C>0$ such that every element $x\in E$ admits a
decomposition of the form $x=u-v$ where $u,v\in E_{+}$ and $\left\Vert
u\right\Vert ,\left\Vert v\right\Vert \leq C\left\Vert x\right\Vert .$

$(b)$ The dual space of $E,$ $E^{\ast},$ when endowed with the dual cone
\[
E_{+}^{\ast}=\left\{  x^{\ast}\in E^{\ast}:x^{\ast}(x)\geq0\text{ for all
}x\in E_{+}\right\}
\]
is a regularly ordered Banach space.

$(c)\ x\leq y$ in $E$ is equivalent to $x^{\ast}(x)\leq x^{\ast}(y)$ for all
$x^{\ast}\in E_{+}^{\ast}.$

$(d)$ $\left\Vert x\right\Vert =\sup\left\{  x^{\ast}(x):x^{\ast}\in
E_{+}^{\ast},\text{ }\left\Vert x^{\ast}\right\Vert \leq1\right\}  $ for all
$x\in E_{+}.$

$(e)$ If $(x_{n})_{n}$ is a decreasing sequence of positive elements of $E$
which converges weakly to $0,$ then $\left\Vert x_{n}\right\Vert
\rightarrow0.$
\end{lemma}

The assertion $(e)$ is a generalization of Dini's lemma in real analysis; see
\cite{CN2014}, p. 173.

\begin{proof}
The assertion $(a)$ follows immediately from Lemma \ref{Lem0}. For $(b)$, see
Davies \cite{Davies1968}, Lemma 2.4. The assertion $(c)$ is an easy
consequence of the Hahn-Banach separation theorem; see \cite{NP2018}, Theorem
2.5.3, p. 100.

The assertion $(d)$ is also a consequence of the Hahn-Banach separation
theorem; see \cite{SW1999}, Theorem 4.3, p. 223.
\end{proof}

\begin{corollary}
\label{cor1}Every ordered Banach space $E$ can be embedded into a space
$C(X)$, where $X$ is a suitable compact space.
\end{corollary}

\begin{proof}
According to the Alaoglu theorem, the set $X=\left\{  x^{\ast}\in E_{+}^{\ast
}:\left\Vert x^{\ast}\right\Vert \leq1\right\}  $ is compact relative to the
$w^{\ast}$ topology. Taking into account the assertions $(c)$ and $(d)$ of
Lemma \ref{lem1} one can easily conclude that $E$ embeds into $C(X)$
(algebraically, isometrically and in order) via the map
\[
\Phi:E\rightarrow C(X),\text{\quad}\left(  \Phi(x)\right)  (x^{\ast})=x^{\ast
}(x).
\]

\end{proof}

The following important result is due to V. Klee \cite{Klee}. A simple proof
of it is available in \cite{NO2020} and \cite{NO2020bis}.

\begin{lemma}
\label{lem1*}Every positive linear operator $T:E\rightarrow F$ acting on
ordered Banach spaces is continuous.
\end{lemma}

Sometimes, spaces with a richer structure are necessary.

A vector lattice is any ordered vector space $E$ such that $\sup\{x,y\}$ and
$\inf\{x,y\}$ exist for all $x,y\in E.$ In this case for each $x\in E$ we can
define $x^{+}=\sup\left\{  x,0\right\}  $ (the positive part of $x$),
$x^{-}=\sup\left\{  -x,0\right\}  $ (the negative part of $x$) and $\left\vert
x\right\vert =\sup\left\{  -x,x\right\}  $ (the modulus of $x$). $\ $We have
$x=x^{+}-x^{-}$ and $\left\vert x\right\vert =x^{+}+x^{-}.$ A vector lattice
endowed with a norm $\left\Vert \cdot\right\Vert $ such that%
\[
\left\vert x\right\vert \leq\left\vert y\right\vert \text{ implies }%
x\leq\left\Vert y\right\Vert
\]
is called a normed vector lattice; it is called a Banach lattice when in
addition it is metric complete.

Examples of Banach lattices are numerous: the discrete spaces $\mathbb{R}%
^{n},$ $c_{0},$ $c$ and $\ell^{p}$ for $1\leq p\leq\infty$ (endowed with the
coordinate-wise order), and the function spaces $C(K)$ (for $K$ a compact
Hausdorff space) and $L^{p}(\mu)$ with $1\leq p\leq\infty$ (endowed with
pointwise order). Of a special interest are the Banach lattices with
\emph{order continuous norm}, that is, the Banach lattices for which every
monotone and order bounded sequence is convergent in the norm topology. So are
$\mathbb{R}^{n},$ $c_{0}$ and $L^{p}(\mu)$ for $1\leq p<\infty$.

\begin{lemma}
\label{lem2}Every monotone and order bounded sequence of elements in a Banach
lattice $E$ with order continuous norm admits a supremum and an infimum and
all closed order intervals in $E$ are weakly compact.
\end{lemma}

For details, see Meyer-Nieberg \cite{MN}, Theorem 2.4.2, p. 86.

\section{The Choquet-Bochner Integral}

This section is devoted to the extension of Choquet's theory of integrability
to the framework with respect to a monotone set function with values in the
positive cone of a regularly ordered Banach space $E$. This draws a parallel
to the real-valued case already treated in full details by Denneberg
\cite{Denn} and Wang and Klir \cite{WK}.

Given a nonempty set $X,$ by a \emph{lattice} of subsets of $X$ we mean any
collection $\Sigma$ of subsets that contains $\emptyset$ and $X$ and is closed
under finite intersections and unions. A lattice $\Sigma$ is an \emph{algebra}
if in addition it is closed under complementation. An algebra which is closed
under countable unions and intersections is called a $\sigma$-algebra.

Of a special interest is the case where $X$ is a compact Hausdorff space and
$\Sigma$ is either the lattice $\Sigma_{up}^{+}(X),$ of all \emph{upper
contour} \emph{closed sets} $S=\left\{  x\in X:f(x)\geq t\right\}  ,$ or the
lattice $\Sigma_{up}^{-}(X)$ of all \emph{upper contour open sets} $S=\left\{
x\in X:f(x)>t\right\}  ,$ associated to pairs $f\in C(X)$ and $t\in
\mathbb{R}.$

When $X$ is a compact metrizable space, $\Sigma_{up}^{+}(X)$ coincides with
the lattice of all closed subsets of $X$ (and $\Sigma_{up}^{-}(X)$ coincides
with the lattice of all open subsets of $X)$.

In what follows $\Sigma$ denotes a lattice of subsets of an abstract set $X$
and $E$ is a regularly ordered Banach space.

\begin{definition}
\label{defcap}A set function $\boldsymbol{\mu}:\Sigma\rightarrow E_{+}$ is
called a vector capacity if it verifies the following two conditions:

$(C1)$ $\boldsymbol{\mu}(\emptyset)=0;$ and

$(C2)~\boldsymbol{\mu}(A)\leq\boldsymbol{\mu}(B)$ for all $A,B\in\Sigma$ with
$A\subset B$.
\end{definition}

Notice that any vector capacity $\boldsymbol{\mu}$ is positive and takes
values in the order interval $[0,$ $\boldsymbol{\mu}(X)].$

An important class of vector capacities is that of \emph{additive}
(respectively $\sigma$-\emph{additive) vector measures with positive values},
that is, of capacities $\boldsymbol{\mu}:\Sigma\rightarrow E_{+}$ with the
property
\[
\boldsymbol{\mu}\left(
{\displaystyle\bigcup\nolimits_{n}}
A_{n}\right)  =%
{\displaystyle\sum\nolimits_{n=1}^{\infty}}
\boldsymbol{\mu}(A_{n}),
\]
for every finite (respectively infinite) sequence $A_{1},A_{2},A_{3},...$ of
disjoint sets belonging to $\Sigma$ such that $\cup_{n}A_{n}\in\Sigma.$

Some other classes of capacities exhibiting various extensions of the property
of additivity are listed below.

A vector capacity $\boldsymbol{\mu}:\Sigma\rightarrow E_{+}$ is called
\emph{submodular} if%
\begin{equation}
\boldsymbol{\mu}(A\cup B)+\boldsymbol{\mu}(A\cap B)\leq\boldsymbol{\mu
}(A)+\boldsymbol{\mu}(B)\text{\quad for all }A,B\in\Sigma\label{submod}%
\end{equation}
and it is called \emph{supermodular} when the inequality (\ref{submod}) works
in the reversed way. Every additive measure taking values in $E_{+}$ is both
submodular and supermodular.

A vector capacity $\boldsymbol{\mu}:\Sigma\rightarrow E_{+}$ is called
\emph{lower continuous\ }(or continuous by ascending sequences)\emph{ }if%
\[
\lim_{n\rightarrow\infty}\boldsymbol{\mu}(A_{n})=\boldsymbol{\mu}(%
{\displaystyle\bigcup\nolimits_{n=1}^{\infty}}
A_{n})
\]
for every nondecreasing sequence $(A_{n})_{n}$ of sets in $\Sigma$ such that
$\cup_{n=1}^{\infty}A_{n}\in\Sigma;$ $\boldsymbol{\mu}$ is called \emph{upper
continuous\ }(or continuous by descending sequences) if
\[
\lim_{n\rightarrow\infty}\boldsymbol{\mu}(A_{n})=\boldsymbol{\mu}\left(
\cap_{n=1}^{\infty}A_{n}\right)
\]
for every nonincreasing sequence $(A_{n})_{n}$ of sets in $\Sigma$ such that
$\cap_{n=1}^{\infty}A_{n}\in\Sigma.$ If $\boldsymbol{\mu}$ is an additive
capacity defined on a $\sigma$-algebra, then its upper/lower continuity is
equivalent to the property of $\sigma$-additivity.

When $\Sigma$ is an algebra of subsets of $X,$ then to each vector capacity
$\boldsymbol{\mu}$ defined on $\Sigma$, one can attach a new vector capacity
$\overline{\boldsymbol{\mu}}$, the \emph{dual} of $\boldsymbol{\mu},$ which is
defined by the formula
\[
\overline{\boldsymbol{\mu}}(A)=\boldsymbol{\mu}(X)-\boldsymbol{\mu}(X\setminus
A).
\]

Notice that $\overline{\left(  \overline{\boldsymbol{\mu}}\right)
}=\boldsymbol{\mu}.$

The dual of a submodular (supermodular) capacity is a supermodular
(submodular) capacity. Also, dual of a lower continuous (upper continuous)
capacity is an upper continuous (lower continuous) capacity.

\begin{example}
There are several standard procedures to attach to a $\sigma$-additive vector
measure $\boldsymbol{\mu}:\Sigma\rightarrow E_{+}$ certain not necessarily
additive capacities. So is the case of \emph{distorted measures,}
$\nu(A)=T(\boldsymbol{\mu}(A)),$ obtained from $\boldsymbol{\mu}$ by applying
to it a continuous nondecreasing distortion $T:[0,\boldsymbol{\mu
}(X)]\rightarrow\lbrack0,\boldsymbol{\mu}(X)].$ The vector capacities $\nu$ so
obtained are both upper and lower continuous.

For example, this is the case when $E=\operatorname*{Sym}(n,\mathbb{R})$,
$\boldsymbol{\mu}(X)=\mathrm{I}$ $($the identity of $\mathbb{R}^{n}),$ and
$T:[0,\mathrm{I}]\rightarrow\lbrack0,\mathrm{I}]$ is the distortion defined by
the formula $T(A)=A^{2}.$
\end{example}

Taking into account the assertions $(c)$ and $(d)$ of Lemma \ref{lem1}, some
aspects (but not all) of the theory of vector capacities are straightforward
consequences of the theory of $\mathbb{R}_{+}$-valued capacities.

\begin{lemma}
\label{lemweak}A set function $\boldsymbol{\mu}:\Sigma\rightarrow E_{+}$ is a
submodular \emph{(}supermodular, lower continuous, upper continuous\emph{)}
vector capacity if and only if $x^{\ast}\circ\boldsymbol{\mu}$ is a submodular
\emph{(}supermodular, lower continuous, upper continuous\emph{)}
$\mathbb{R}_{+}$-valued capacity whenever $x^{\ast}\in E_{+}^{\ast}.$

When $E=\mathbb{R}^{n},$ this assertion can be formulated via the components
$\mu_{k}=\operatorname*{pr}_{k}\circ\boldsymbol{\mu}$ of $\mathbf{\mu}.$
\end{lemma}

In what follows the term of (upper) \emph{measurable} \emph{function} refers
to any function $f:X\rightarrow\mathbb{R}$ whose all upper contour sets
$\left\{  x\in X:f(x)\geq t\right\}  $ belong to $\Sigma$. When $\Sigma$ is a
$\sigma$-algebra, this notion of measurability is equivalent to the Borel measurability.

We will denote by $B(\Sigma)$ the set of all bounded measurable functions
$f:X\rightarrow\mathbb{R}.$ In general $B(\Sigma)$ is not a vector space
(unless the case when $\Sigma$ is a $\sigma$-algebra). However, even when
$\Sigma$ is only an algebra, the set $B(\Sigma)$ plays some nice properties of
stability: if $f,g\in B(\Sigma)$ and $\alpha,\beta\in\mathbb{R},$ then
\[
\inf\{f,g\},\text{ }\sup\left\{  f,g\right\}  \text{ and }\alpha+\beta f
\]
also belong to $B(\Sigma).$ See \cite{MM2004}, Proposition 15.

Given a capacity $\mu:\Sigma\rightarrow\mathbb{R}_{+},$ the \emph{Choquet
integral} of a measurable function $f:X\rightarrow\mathbb{R}$ on a set
$A\in\Sigma$ is defined as the sum of two Riemann improper integrals,
\begin{align}
(\operatorname*{C})\int_{A}f\mathrm{d}\mu &  =\int_{0}^{+\infty}\mu\left(
\{x\in A:f(x)\geq t\}\right)  \mathrm{d}t\label{Chint}\\
&  +\int_{-\infty}^{0}\left[  \mu\left(  \{x\in A:f(x)\geq t\}\right)
-\mu(A)\right]  \mathrm{d}t.\nonumber
\end{align}
Accordingly, $f$ is said to be \emph{Choquet integrable} on $A$ if both
integrals above are finite. See the seminal paper of Choquet \cite{Ch1954}.

If $f\geq0$, then the last integral in the formula (\ref{Chint}) is 0. When
$\Sigma$ is a $\sigma$-algebra, the inequality sign $\geq$ in the above two
integrands can be replaced by $>;$ see \cite{WK}, Theorem 11.1,\emph{ }p. 226.

Every bounded measurable function is Choquet integrable. The Choquet integral
coincides with the Lebesgue integral when the underlying set function $\mu$ is
a $\sigma$-additive measure defined on a $\sigma$-algebra.

The theory of Choquet integral is available from numerous sources including
the books of Denneberg \cite{Denn}, and Wang and Klir \cite{WK}.

The concept of integrability of a measurable function with respect to a vector
capacity $\boldsymbol{\mu}:\Sigma\rightarrow E_{+}$ can be introduced by a
fusion between the Choquet integral and the Bochner theory of integration of
vector-valued functions.

Recall that a function $\psi:\mathbb{R}\rightarrow E$ is \emph{Bochner
integrable} with respect to the Lebesgue measure on $\mathbb{R}$ if there
exists a sequence of step functions $\psi_{n}:\mathbb{R}\rightarrow E$ such
that%
\[
\lim_{n\rightarrow\infty}\psi_{m}(t)=\psi(t)\text{ almost everywhere and }%
\int_{\mathbb{R}}\left\Vert \psi-\psi_{n}\right\Vert \mathrm{d}t\rightarrow0.
\]
In this case, the (Bochner) integral of $\psi$ is defined by
\[
\int_{\mathbb{R}}\psi\mathrm{d}t=\lim_{n\rightarrow\infty}\int_{\mathbb{R}%
}\psi_{n}\mathrm{d}t.
\]
Notice that if $T:E\rightarrow F$ is a bounded linear operator, then
\begin{equation}
T\left(  \int_{\mathbb{R}}\psi\mathrm{d}t\right)  =\int_{\mathbb{R}}T\circ
\psi\mathrm{d}t. \label{CB_opercommuting}%
\end{equation}

Details about Bochner integral are available in the books of Diestel and Uhl
\cite{DU1977} and Dinculeanu \cite{Din}.

\begin{definition}
\label{defIntB}A measurable function $f:X\rightarrow\mathbb{R}$ is called
Choquet-Bochner integrable with respect to the vector capacity
$\boldsymbol{\mu}:\Sigma\rightarrow E_{+}$ on the set $A\in\Sigma$ if for
every $A\in\Sigma$ the functions $t\rightarrow\boldsymbol{\mu}\left(  \{x\in
A:f(x)\geq t\}\right)  $ and $t\rightarrow\boldsymbol{\mu}\left(  \{x\in
A:f(x)\geq t\}\right)  -\boldsymbol{\mu}(A)$ are Bochner integrable
respectively on $[0,\infty)$ and $(-\infty,0].$ Under these circumstances, the
Choquet-Bochner integral over $A$ is defined by the formula%
\begin{align*}
(\operatorname*{CB})\int_{A}f\mathrm{d}\boldsymbol{\mu}  &  =\int_{0}%
^{+\infty}\boldsymbol{\mu}\left(  \{x\in A:f(x)\geq t\}\right)  \mathrm{d}t\\
+  &  \int_{-\infty}^{0}\left[  \boldsymbol{\mu}\left(  \{x\in A:f(x)\geq
t\}\right)  -\mu(A)\right]  \mathrm{d}t.
\end{align*}

\end{definition}

According to the formula (\ref{CB_opercommuting}), if $f$ is Choquet-Bochner
integrable, then
\begin{equation}
x^{\ast}\left(  (\operatorname*{CB})\int_{A}f\mathrm{d}\boldsymbol{\mu
}\right)  =(\operatorname*{C})\int_{A}f\mathrm{d}(x^{\ast}\circ\mu),
\label{functcomm}%
\end{equation}
for every positive linear functional $x^{\ast}\in E^{\ast}.$

A large class of Choquet-Bochner integrable is indicated below.

\begin{lemma}
\label{lem6}If $f:X\rightarrow\mathbb{R}$ is a bounded measurable function,
then it is Choquet-Bochner integrable on every set $A\in\Sigma$.
\end{lemma}

\begin{proof}
Suppose that $f$ takes values in the interval $[0,M].$ Then the function
$\varphi(t)=\boldsymbol{\mu}\left(  \{x\in A:f(x)\geq t\}\right)  $ is
positive and nonincreasing on the interval $[0,M]$ and null outside this
interval. As a consequence, the function $\varphi$ is the uniform limit of the
sequence of step functions defined as follows:%
\begin{align*}
\varphi_{n}(t)  &  =0\text{ if }t\notin\lbrack0,M]\\
\varphi_{n}(t)  &  =\varphi(t)\text{ if }t=M
\end{align*}
and
\[
\varphi_{n}(t)=\varphi(\frac{k}{n}M)
\]
if $t\in\lbrack\frac{k}{n}M,\frac{k+1}{n}M)$ and $k=0,...,n-1.$ A simple
computation shows that%
\begin{align*}
\int_{0}^{+\infty}\left\Vert \varphi(t)-\varphi_{n}(t)\right\Vert \mathrm{d}t
&  =\int_{0}^{M}\left\Vert \varphi(t)-\varphi_{n}(t)\right\Vert \mathrm{d}t\\
&  \leq%
{\displaystyle\sum\nolimits_{l=0}^{n-1}}
\int_{t_{k}}^{t_{k+1}}\left\Vert \varphi(t)-\varphi_{n}(t)\right\Vert
\mathrm{d}t\\
&  \leq\frac{2M}{n}\boldsymbol{\mu}(X)\rightarrow0
\end{align*}
as $n\rightarrow\infty,$ which means the Bochner integrability of the function
$\varphi.$ See \cite{DU1977}, p. 44.

The other cases, when $f$ takes values in the interval $[M,0]$ or in an
interval $[m,M]$ with $m<0<M,$ can be treated in a similar way.
\end{proof}

The next lemma collects a number of simple (but important) properties of the
Choquet-Bochner integral.

\begin{lemma}
\label{lemgenprop}Suppose that $E$ is an ordered Banach space,
$\boldsymbol{\mu}:\Sigma\rightarrow E_{+}$ is a vector capacity and
$A\in\Sigma$.

$(a)$ If $f$ and $g$ are Choquet-Bochner integrable functions, then%
\begin{gather*}
f\geq0\text{ implies }(\operatorname*{CB})\int_{A}f\mathrm{d}\boldsymbol{\mu
}\geq0\text{ \quad\emph{(}positivity\emph{)}}\\
f\leq g\text{ implies }\left(  \operatorname*{CB}\right)  \int_{A}%
f\mathrm{d}\boldsymbol{\mu}\leq\left(  \operatorname*{CB}\right)  \int
_{A}g\mathrm{d}\boldsymbol{\mu}\text{ \quad\emph{(}monotonicity\emph{)}}\\
\left(  \operatorname*{CB}\right)  \int_{A}af\mathrm{d}\boldsymbol{\mu}%
=a\cdot\left(  \operatorname*{CB}\right)  \int_{A}f\mathrm{d}\boldsymbol{\mu
}\text{\quad for all }a\geq0\text{ \quad\emph{(}positive\emph{ }%
homogeneity\emph{)}}\\
\left(  \operatorname*{CB}\right)  \int_{A}1\cdot\mathrm{d}\boldsymbol{\mu
}=\boldsymbol{\mu}(A)\text{\quad\emph{(}calibration\emph{).}}%
\end{gather*}
$(b)$ In general, the Choquet-Bochner integral is not additive but if $f$ and
$g$ are Choquet-Bochner integrable functions and also comonotonic in the sense
of Dellacherie \emph{\cite{Del1970}) (}that is, $(f(\omega)-f(\omega^{\prime
}))\cdot(g(\omega)-g(\omega^{\prime}))\geq0$, for all $\omega,\omega^{\prime
}\in X$\emph{), }then
\[
\left(  \operatorname*{CB}\right)  \int_{A}(f+g)\mathrm{d}\boldsymbol{\mu
}=\left(  \operatorname*{CB}\right)  \int_{A}f\mathrm{d}\boldsymbol{\mu
}+\left(  \operatorname*{CB}\right)  \int_{A}g\mathrm{d}\boldsymbol{\mu}.
\]
In particular, the Choquet-Bochner integral is translation invariant,%
\[
\left(  \operatorname*{CB}\right)  \int_{A}(f+c)\mathrm{d}\boldsymbol{\mu
}=\left(  \operatorname*{CB}\right)  \int_{A}f\mathrm{d}\boldsymbol{\mu
}+c\boldsymbol{\mu}(A),
\]
for all Choquet-Bochner integrable functions $f$, all sets $A\in\Sigma$ and
all numbers $c\in\mathbb{R}$.
\end{lemma}

\begin{proof}
According to Lemma \ref{lem1} $(c)$, and formula (\ref{CB_opercommuting}),
applied in the case of an arbitrary functional $x^{\ast}\in E_{+}^{\ast},$ the
proof of both assertions $(a)$ and $(b)$ reduce to the case of capacities with
values in $\mathbb{R}_{+},$ already covered by Proposition 5.1 in \cite{Denn},
pp. 64-65.
\end{proof}

\begin{corollary}
The equality%
\[
\left(  \operatorname*{CB}\right)  \int_{A}(\alpha f+c)\mathrm{d}%
\boldsymbol{\mu}=\alpha\left(  \left(  \operatorname*{CB}\right)  \int
_{A}f\mathrm{d}\boldsymbol{\mu}\right)  +c\cdot\boldsymbol{\mu}(A),
\]
holds for all Choquet-Bochner integrable functions $f$, all sets $A\in\Sigma$
and all numbers $\alpha\in\mathbb{R}_{+}$ and $c\in\mathbb{R}$.
\end{corollary}

The next result describes how the special properties of a vector capacity
transfer to the Choquet-Bochner integral.

\begin{theorem}
\label{thmtransfer}$(a)$ If $\boldsymbol{\mu}$ is an upper continuous
capacity, then the Choquet-Bochner integral is an upper continuous operator,
that is,%
\[
\lim_{n\rightarrow\infty}\left\Vert \left(  \operatorname*{CB}\right)
\int_{A}f_{n}\mathrm{d}\boldsymbol{\mu}-\left(  \operatorname*{CB}\right)
\int_{A}f\mathrm{d}\boldsymbol{\mu}\right\Vert =0,
\]
whenever $(f_{n})_{n}$ is a nonincreasing sequence of Choquet-Bochner
integrable functions that converges pointwise to the Choquet-Bochner
integrable function $f$ and $A\in\Sigma$.

$(b)$ If $\boldsymbol{\mu}$ is a lower continuous capacity, then the
Choquet-Bochner integral is lower continuous in the sense that%
\[
\lim_{n\rightarrow\infty}\left\Vert \left(  \operatorname*{CB}\right)
\int_{A}f_{n}\mathrm{d}\boldsymbol{\mu}-\left(  \operatorname*{CB}\right)
\int_{A}f\mathrm{d}\boldsymbol{\mu}\right\Vert =0
\]
whenever $(f_{n})_{n}$ is a nondecreasing sequence of Choquet-Bochner
integrable functions that converges pointwise to the Choquet-Bochner
integrable function $f$ and $A\in\Sigma$.

$(c)$ If $\Sigma$ is an algebra and $\boldsymbol{\mu}:\Sigma\rightarrow E_{+}$
is a submodular capacity, then the Choquet-Bochner integral is a submodular
operator in the sense that
\[
\left(  \operatorname*{CB}\right)  \int_{A}\sup\left\{  f,g\right\}
\mathrm{d}\boldsymbol{\mu}+\left(  \operatorname*{CB}\right)  \int_{A}%
\inf\{f,g\}\mathrm{d}\boldsymbol{\mu}\leq\left(  \operatorname*{CB}\right)
\int_{A}f\mathrm{d}\mu+(\operatorname*{CB})\int_{A}g\mathrm{d}\boldsymbol{\mu}%
\]
whenever $f$ and $g$ are Choquet-Bochner integrable and $A\in\Sigma$ .
\end{theorem}

\begin{proof}
$(a)$ Since $\mathbf{\mu}$ is an upper continuous capacity and $(f_{n})_{n}$
is a nonincreasing sequence of measurable functions that converges pointwise
to the measurable function $f$ it follows that
\[
\boldsymbol{\mu}\left(  \{x\in A:f_{n}(x)\geq t\}\right)  \searrow
\boldsymbol{\mu}\left(  \{x\in A:f(x)\geq t\}\right)
\]
in the norm topology. Taking into account the property of monotonicity of the
Choquet-Bochner integral (already noticed in Lemma \ref{lemgenprop} $(a))$ we
have%
\[
(\operatorname*{CB})\int_{A}f_{n}\mathrm{d}\boldsymbol{\mu}\geq
(\operatorname*{CB})\int_{A}f_{n}\mathrm{d}\boldsymbol{\mu}\geq\mathbf{\cdots
}\geq(\operatorname*{CB})\int_{A}f\mathrm{d}\boldsymbol{\mu}\mathbf{,}%
\]

so by Bepo Levi's monotone convergence theorem from the theory of Lebesgue
integral (see \cite{Din}, Theorem 2, p. 133) it follows that
\[
x^{\ast}\left(  (\operatorname*{CB})\int_{A}f_{n}\mathrm{d}\boldsymbol{\mu
}\right)  =(\operatorname*{C})\int_{A}f_{n}\mathrm{d}\mu^{\ast}\rightarrow
(\operatorname*{C})\int_{A}f\mathrm{d}\mu^{\ast}=x^{\ast}\left(
(\operatorname*{CB})\int_{A}f\mathrm{d}\boldsymbol{\mu}\right)
\]
for all $x^{\ast}\in E_{+}^{\ast}$. The conclusion of the assertion $(c)$ is
now a direct consequence of the generalized Dini's lemma (see Lemma \ref{lem1}
$(e)$).

$(b)$ The argument is similar to that used to prove the assertion $(a)$.

$(c)$ Since $\Sigma$ is an algebra, both functions $\inf\{f,g\}$ and
$\sup\left\{  f,g\right\}  $ are measurable. The fact that $\mathbf{\mu}$ is
submodular implies
\begin{multline*}
\boldsymbol{\mu}\left(  \left\{  x:\sup\{f,g\}(x)\geq t\right\}  \right)
+\boldsymbol{\mu}\left(  \left\{  x:\inf\{f,g\}(x)\geq t\right\}  \right) \\
=\boldsymbol{\mu}\left(  \left\{  x:f(x)\geq t\right\}  \cup\left\{
x:g(x)\geq t\right\}  \right)  +\boldsymbol{\mu}\left(  \left\{  x:f(x)\geq
t\right\}  \cap\left\{  x:g(x)\geq t\right\}  \right) \\
\leq\boldsymbol{\mu}\left(  \left\{  x:f(x)\geq t\right\}  \right)
+\boldsymbol{\mu}\left(  \left\{  x:g(x)\geq t\right\}  \right)  ,
\end{multline*}
and the same works when $\boldsymbol{\mu}$ is replaced by $\mu^{\ast}=x^{\ast
}\circ\boldsymbol{\mu},$ where $x^{\ast}\in E_{+}^{\ast}$ is arbitrarily
fixed. Integrating side by side the last inequality it follows that%
\[
\left(  \operatorname*{C}\right)  \int_{A}\sup\left\{  f,g\right\}
\mathrm{d}\mu^{\ast}+\left(  \operatorname*{C}\right)  \int_{A}\inf
\{f,g\}\mathrm{d}\mu^{\ast}\leq\left(  \operatorname*{C}\right)  \int
_{A}f\mathrm{d}\mu^{\ast}+(\operatorname*{C})\int_{A}g\mathrm{d}\mu^{\ast},
\]
which yields (via Lemma \ref{lem1} $(c))$ the submodularity of the
Choquet-Bochner integral.
\end{proof}

The property of subadditivity of the Choquet-Bochner integral makes the
objective of the following result:

\begin{theorem}
\label{subadthm}\emph{(The Subadditivity Theorem) }If $\boldsymbol{\mu}$ is a
submodular capacity, then the associated Choquet-Bochner integral is
subadditive, that is,%
\[
\left(  \operatorname*{CB}\right)  \int_{A}(f+g)\mathrm{d}\boldsymbol{\mu}%
\leq\left(  \operatorname*{CB}\right)  \int_{A}f\mathrm{d}\boldsymbol{\mu
}+\left(  \operatorname*{CB}\right)  \int_{A}g\mathrm{d}\boldsymbol{\mu}%
\]
whenever $f$, $g$ and $f+g$ are Choquet-Bochner integrable functions and
$A\in\Sigma$.

In addition, when $E$ is a Banach lattice and $f$, $g,$ $f-g$ and $g-f$ are
Choquet-Bochner integrable functions, then the following integral analog of
the modulus inequality holds true:
\[
\left\vert (\operatorname*{CB})\int_{A}f\mathrm{d}\boldsymbol{\mu
}-(\operatorname*{CB})\int_{A}g\mathrm{d}\boldsymbol{\mu}\right\vert
\leq(\operatorname*{CB})\int_{A}|f-g|\mathrm{d}\mu
\]
for all $A\in\Sigma$. In particular,
\[
\left\vert (\operatorname*{CB})\int_{A}f\mathrm{d}\boldsymbol{\mu}\right\vert
\leq(\operatorname*{CB}\int_{A}|f|\mathrm{d}\boldsymbol{\mu}%
\]
whenever $f$ and $-f$ are Choquet-Bochner integrable functions and $A\in
\Sigma.$
\end{theorem}

\begin{proof}
According to Lemma \ref{lemweak}, if $\boldsymbol{\mu}$ is submodular, then
every real-valued capacity $\mu^{\ast}=x^{\ast}\circ\boldsymbol{\mu}$ also is
submodular, whenever $x^{\ast}\in E_{+}^{\ast}$. Then
\[
\left(  \operatorname*{C}\right)  \int_{A}f\mathrm{d}\mu^{\ast}+\left(
\operatorname*{C}\right)  \int_{A}g\mathrm{d}\mu^{\ast}\leq\left(
\operatorname*{C}\right)  \int_{A}(f+g)\mathrm{d}\mu^{\ast},
\]
as a consequence of Theorem\emph{ }6.3, p. 75, in \cite{Denn}. Therefore the
first inequality in the statement of Theorem \ref{subadthm} is now a direct
consequence of Lemma \ref{lem1} $(c)$ and the formula (\ref{functcomm}).

For the second inequality, notice that the subadditivity property implies
\begin{align*}
(\operatorname*{CB})\int_{A}f\mathrm{d}\boldsymbol{\mu}  &
=(\operatorname*{CB})\int_{A}\left(  f-g+g\right)  \mathrm{d}\boldsymbol{\mu
}\\
&  \leq\left(  \operatorname*{CB}\right)  \int_{A}\left(  f-g\right)
\mathrm{d}\boldsymbol{\mu}+\left(  \operatorname*{CB}\right)  \int
_{A}g\mathrm{d}\boldsymbol{\mu}\mathbf{\boldsymbol{,}}%
\end{align*}
and taking into account that $f-g\leq\left\vert f-g\right\vert $ we infer that%
\[
(\operatorname*{CB})\int_{A}f\mathrm{d}\boldsymbol{\mu}-\left(
\operatorname*{CB}\right)  \int_{A}g\mathrm{d}\boldsymbol{\mu}\leq\left(
\operatorname*{CB}\right)  \int_{A}\left\vert f-g\right\vert \mathrm{d}%
\boldsymbol{\mu}.
\]
Interchanging $f$ and $g$ we also obtain%
\[
\pm((\operatorname*{CB})\int_{A}f\mathrm{d}\boldsymbol{\mu}-\left(
\operatorname*{CB}\right)  \int_{A}g\mathrm{d}\boldsymbol{\mu}%
\mathbf{\boldsymbol{)}}\leq\left(  \operatorname*{CB}\right)  \int
_{A}\left\vert f-g\right\vert \mathrm{d}\boldsymbol{\mu}%
\]
and the proof is done.
\end{proof}

\section{The integral representation of Choquet operators defined on a space
$C(X)$}

The special case when $X$ is a compact Hausdorff space and $\Sigma
=\mathcal{B}(X),$ the $\sigma$-algebra of all Borel subsets of $X,$ allows us
to shift the entire discussion concerning the Choquet-Bochner integral from
the vector space $B(\Sigma)$ (of all real-valued Borel measurable functions$)$
to the Banach lattice $C(X)$ (of all real-valued continuous functions defined
on $X$. Indeed, in this case $C(X)$ is a subspace of $B(\Sigma)$ and all nice
properties stated in Lemma \ref{lemgenprop} and in Theorem \ref{subadthm}
remain true when restricting the integral to the space $C(X).$ As a
consequence the integral operator%
\begin{equation}
\operatorname*{CB}\nolimits_{\boldsymbol{\mu}}:C(X)\rightarrow E,\text{\quad
}\operatorname*{CB}\nolimits_{\boldsymbol{\mu}}(f)=(\operatorname*{CB}%
)\int_{X}f\mathrm{d}\boldsymbol{\mu,} \label{CPoper}%
\end{equation}
associated to a vector-valued submodular capacity $\boldsymbol{\mu
}:\mathcal{B}(X)\rightarrow E,$ is a Choquet operator.

The linear and positive functionals defined on $C(X)$ can be represented
either as integrals with respect to a unique regular Borel measure (the
Riesz-Kakutani representation theorem) or as Choquet integrals (see Epstein
and Wang \cite{EW1996}).

\begin{example}
The space $c$ of all convergent sequences of real numbers can be identified
with the space of continuous functions on $\mathbb{\hat{N}}=\mathbb{N}%
\cup\{\infty\}$ $($the one-point compactification of the discrete space
$\mathbb{N)}$. The functional
\[
I:c\rightarrow\mathbb{R},\text{\quad}I(x)=\lim_{n\rightarrow\infty}x(n)
\]
is linear and monotone \emph{(}therefore continuous by Lemma \emph{\ref{lem1*}%
)} and its Riesz representation is%
\[
I(x)=\int_{\mathbb{\hat{N}}}x(n)\mathrm{d}\delta_{\infty}(n),
\]
where $\delta_{\infty}$ is the Dirac measure concentrated at $\infty,$ that
is, $\delta_{\infty}(A)=1$ if $A\in\mathcal{P}(\mathbb{\hat{N}})$ and
$\{\infty\}\in A$ and $\delta_{\infty}(A)=0$ if $\{\infty\}\notin A.$
Meantime, $I$ admits the Choquet representation%
\[
I(x)=(\operatorname*{C})\int x(n)d\mu(n)
\]
where $\mu$ is the capacity defined on the power set $\mathcal{P}%
(\mathbb{\hat{N}})$ by $\mu(A)=0$ if $A$ is finite and $\mu(A)=1$ otherwise.
\end{example}

As the Riesz-Kakutani representation theorem also holds in the case of linear
and positive operators $T:C(X)\rightarrow E$, it seems natural to search for
an analogue in the case of Choquet operators. The answer is provided by the
following representation theorem that extends a result due to Epstein and Wang
\cite{EW1996} from functionals to operators:

\begin{theorem}
\label{thmEWgen}Let $X$ be a compact Hausdorff space and $E$ be a Banach
lattice with order continuous norm. Then for every comonotonic additive and
monotone operator $I:C(X)\rightarrow E$ with $I(1)>0$ there exists a unique
upper continuous vector capacity $\boldsymbol{\mu}:\Sigma_{up}^{+}%
(X)\rightarrow E_{+}$ such that%
\begin{equation}
I(f)=(\operatorname*{CB})\int_{X}f\mathrm{d}\boldsymbol{\mu}\text{\quad for
all }f\in C(X). \label{reprformula}%
\end{equation}
Moreover, $\boldsymbol{\mu}$\ admits a unique extension to $\mathcal{B}(X)$
$($also denoted $\boldsymbol{\mu})$ that fulfils the following two properties
of regularity:

$(R1)$ $\boldsymbol{\mu}(A)=\sup\left\{  \boldsymbol{\mu}(K):K\text{ closed,
}K\subset A\right\}  $\ for all $A\in\mathcal{B}(X);$

$(R2)$ $\boldsymbol{\mu}(K)=\inf\left\{  \boldsymbol{\mu}(O):O\text{ open,
}O\supset K\right\}  $\ for$\ $all closed sets $K.$
\end{theorem}

Recall that $\Sigma_{up}^{+}(X)$ represents the lattice of upper contour sets
associated to the continuous functions defined on $X.$ This lattice coincides
with the lattice of all closed subsets of $X$ when $X$ is compact and metrizable.

The proof of Theorem \ref{thmEWgen} needs several auxiliary results and will
be detailed at the end of this section.

\begin{lemma}
\label{lem-aux-1*}Let $E$ be an ordered Banach space. Then every monotone and
translation invariant operator $I:C(X)\rightarrow E$ is Lipschitz continuous.
\end{lemma}

\begin{proof}
Given $f,g\in C(X),$ one can choose a decreasing sequence $(\alpha_{n})_{n}$
of positive numbers such that $\alpha_{n}\downarrow\left\Vert f-g\right\Vert
.$ Then $f\leq g+\left\Vert f-g\right\Vert \leq g+\alpha_{n}\cdot1,$ which
implies%
\[
I(f)\leq I(g+\alpha_{n}\cdot1)=I(g)+\alpha_{n}I(1)
\]
due to the properties of monotonicity and translation invariance and positive
homogeneity of $I.$ Since the role of $f$ and $g$ is symmetric, this leads to
the fact that $\left\vert I(f)-I(g)\right\vert \leq I(1)\cdot\alpha_{n}$ for
all $n,$ whence by passing to the limit as $n\rightarrow\infty,$ we conclude
that%
\[
\left\vert I(f)-I(g)\right\vert \leq I(1)\cdot\left\Vert f-g\right\Vert .
\]

\end{proof}

\begin{lemma}
\label{lem-aux-1}Let $E$ be an ordered Banach space. Then every comonotonic
additive and monotone operator $I:C(X)\rightarrow E$ is positively homogeneous.
\end{lemma}

\begin{proof}
Let $f\in C(X)_{+}.$ Since $I$ is comonotonic additive, we get
$I(0)=I(0+0)=2\cdot I(0),$ which implies $I(0)=0$. As a consequence, $I(0\cdot
f)=0=0\cdot I(f).$ Then the same argument shows that $I(2f)=I(f+f)=2I(f)$ and
by mathematical induction we infer that $I(pf)=pI(f)$ for all $p\in\mathbb{N}$.

Now, consider the case of positive rational numbers $r=p/q,$ where
$p,q\in\mathbb{N}$. Then $I(f)=I\left(  q\cdot\frac{1}{q}f\right)  =qI\left(
\frac{1}{q}f\right)  $, which implies $I\left(  \frac{1}{q}f\right)  =\frac
{1}{q}\cdot I(f)$. Therefore
\[
I\left(  \frac{p}{q}f\right)  =pI\left(  \frac{1}{q}f\right)  =\frac{p}%
{q}\cdot I(f).
\]

Passing to the case of an arbitrary positive number $\alpha,$ let us choose a
decreasing sequence $(r_{n})_{n}$ of rationals converging to $\alpha$. Then
$r_{n}f\geq\alpha f$ for all $n,$ which yields $r_{n}I(f)=I(r_{n}f)\geq
I(\alpha f)$ $n\in\mathbb{N}$. Passing here to limit (in the norm of $E$) it
follows $\alpha I(f)\geq I(\alpha f)$. On the other hand, considering a
sequence of positive rational numbers $s_{n}\nearrow\alpha$ and reasoning as
above, we easily obtain $\alpha I(f)\leq I(\alpha f)$, which combined with the
previous inequality proves the assertion of Lemma \ref{lem-aux-1} in the case
of nonnegative functions.

When $f\in C(X)$ is arbitrary, one can choose a positive number $\lambda$ such
that $f+\lambda\geq0.$ By the above reasoning and the property of comonotonic
additivity, for all $\alpha\geq0,$
\begin{align*}
\alpha I(f)+\alpha\lambda I(1)  &  =\alpha I\left(  f+\lambda\right)
=I(\alpha\left(  f+\lambda\right)  )\\
&  =I(\alpha f+\alpha\lambda)=I(\alpha f)+\alpha\lambda I(1),
\end{align*}
which ends the proof of Lemma \ref{lem-aux-1}.
\end{proof}

\begin{lemma}
\label{lem-aux-3}Let $E$ be a Banach lattice with order continuous norm. Then
every monotone, positively homogeneous and translation invariant operator
$I:C(X)\rightarrow E$ is weakly compact and upper continuous.
\end{lemma}

\begin{proof}
The weak compactness of $I$ follows from the fact that $I$ maps the closed
order interval in $C(X)$ (that is, all closed balls) into closed order
intervals in $E$ and all such intervals are weakly compact in $E$. See Lemma 3.

For the property of upper continuity, let $(f_{n})_{n}$ be any nonincreasing
sequence of functions in $C(X)$, which converges pointwise to a continuous
function $f.$ By Dini's lemma, the sequence $(f_{n})_{n}$ is convergent to $f$
in the norm topology of $C(X).$ Therefore, for each $\varepsilon>0$ there is
an index $N$ such that $f\leq f_{n}<f+\varepsilon$ for all $n\geq N.$ Since
$I$ is monotone it follows that%
\begin{equation}
I(f)\leq I(f_{n})\leq I(f+\varepsilon)=I(f)+\varepsilon\cdot I(1)\text{ for
all }n\geq N, \label{consDini}%
\end{equation}
where $1$ is the unit of $C(X).$ Taking into account that $E$ has order
continuous norm and $I(f_{n})\geq I(f_{n+1})\geq I(f)$ for all $n$ it follows
that the limit $\lim_{n\rightarrow\infty}I(f_{n})$ exists in $E$ and
$\lim_{n\rightarrow\infty}I(f_{n})\geq I(f)$. Combining this fact with
(\ref{consDini}) we infer that $I(f)\leq\lim_{n\rightarrow\infty}I(f_{n})\leq
I(f)+\varepsilon I(1).$ Since $\varepsilon>0$ was chosen arbitrarily we
conclude that $\lim_{n\rightarrow\infty}I(f_{n})=I(f).$
\end{proof}

The next result, was stated for real-valued functionals in \cite{Zhou}, Lemma
1 (and attributed by him to Masimo Marinacci). For the convenience of the
reader we include here the details.

\begin{lemma}
\label{lem-aux-2}Let $E$ be an ordered Banach space.

$(a)$ Suppose that $I:C(X)\rightarrow E$ is a monotone, positively homogeneous
and translation invariant operator. The following two properties are equivalent:

$(a_{1})$ $\lim_{n\rightarrow\infty}I(f_{n})=I(f)$ for any nonincreasing
sequence $(f_{n})_{n}$ in $C(X)$ that converges pointwise to a function $f$
also in $C(X);$

$(a_{2})$ $\lim_{n\rightarrow\infty}I(f_{n})\leq I(f)$ for any nonincreasing
sequence $(f_{n})_{n}$ in $C(X)$ and any $f$ in $C(X)$ such that for each
$x\in X$ there is an index $n_{x}\in\mathbb{N}$ such that $f_{n}(x)\leq f(x)$
whenever $n\geq n_{x}$.

$(b)$ For any vector capacity $\boldsymbol{\mu}:\mathcal{B}(X)\rightarrow
E_{+}$, the following two properties are equivalent :

$(b_{1})$ $\lim_{n\rightarrow\infty}\boldsymbol{\mu}(A_{n})=\boldsymbol{\mu
}(A)$, for any nonincreasing sequence $(A_{n})_{n}$ of sets in $\mathcal{B}%
(X)$ such that $A=\cap_{n=1}^{\infty}A_{n};$

$(b_{2})$ $\lim_{n\rightarrow\infty}\boldsymbol{\mu}(A_{n})\leq\boldsymbol{\mu
}(A)$, for any nonincreasing sequence $(A_{n})_{n}$ of sets in $\mathcal{B}%
(X)$ and any $A\in\mathcal{B}(X)$ such that $\cap_{n=1}^{\infty}A_{n}\subset
A$.

All the limits above are considered in the norm topology of $E$.
\end{lemma}

\begin{proof}
$(a_{1})\Rightarrow(a_{2})$. Let an arbitrary sequence $(f_{n})_{n}$ in $C(X)$
and $f\in C(X)$, be such that $(f_{n})_{n}$ is nonincreasing and, for all
$x\in X$, there is an $n_{x}$ with $f_{n}(x)\leq f(x)$ for all $n\geq n_{x}$.
Since the sequence $(\max\{f_{n},f\})$ is also a nonincreasing sequence in
$C(X)$ and $\lim_{n\rightarrow\infty}\max\{f_{n}(x),f(x)\}=f(x)$ for all $x\in
X$, $(a_{i})$ implies $\lim_{n\rightarrow\infty}I(\max\{f_{n},f\})=I(f)$. By
the monotonicity of $I$ it follows $\lim_{n\rightarrow\infty}I(f_{n})\leq
\lim_{n\rightarrow\infty}I(\max\{f_{n},f\})=I(f)$.

$(a_{2})\Rightarrow(a_{1})$. Let $f_{n},f\in C(X)$, $n\in\mathbb{N}$, be such
that $(f_{n})_{n}$ is nonincreasing and $\lim_{n\rightarrow\infty}%
f_{n}(x)=f(x)$, for all $x\in X$. Fix $\varepsilon>0$ arbitrary. Since for all
$x\in X$, there exists $n_{x}\in\mathbb{N}$ such that $f_{n}(x)\leq
f(x)+\varepsilon$, for all $n\geq n_{x}$, $(a_{ii})$ implies (from
monotonicity, positive homogeneity and translation invariance) $\lim
_{n\rightarrow\infty}I(f_{n})\leq I(f+\varepsilon\cdot1)=I(f)+\varepsilon
I(1)$. Passing with $\varepsilon\rightarrow0$ it follows $\lim_{n\rightarrow
\infty}I(f_{n})\leq I(f)$. But since $(f_{n})_{n}$ is nonincreasing and $I$ is
monotone, we also have $\lim_{n\rightarrow\infty}I(f_{n})\geq I(f)$, which
combined with the previous inequality implies $\lim_{n\rightarrow\infty
}I(f_{n})=I(f)$.

The equivalence $(b_{1})\Leftrightarrow(b_{2})$ can be proved in a similar way.
\end{proof}

Recall that in the case of compact Hausdorff space $X$ the lattice
$\Sigma_{up}^{+}(X)$ represents the lattice of upper contour sets associated
to the continuous functions defined on $X.$ This lattice coincides with the
lattice of all closed subsets of $X$ when $X$ is compact and metrizable;
indeed, if $d$ is the metric of $X,$ then every closed subset $A\subset X$
admits the representation $A=\left\{  x:-d(x,A)\geq0\right\}  .$

\begin{proof}
[Proof of Theorem $3$]Notice first that according to Lemma \ref{lem-aux-1} and
Lemma \ref{lem-aux-3} the operator $I$ is also positively homogeneous and
upper continuous.

Every set $K\in\Sigma_{up}^{+}(X)$ admits a representation of the form%
\[
K=\left\{  x:f(x)\geq\alpha\right\}  ,
\]
for suitable $f\in C(X)$ and $\alpha\in\mathbb{R}.$ As a consequence its
characteristic function $\chi_{K}$ is the pointwise limit of a nonincreasing
sequence $(f_{n}^{K})_{n}$ of continuous and nonnegative functions. For
example, one may choose
\begin{equation}
f_{n}^{K}(x)=1-\inf\left\{  1,n(\alpha-f)^{+}\right\}  =\left\{
\begin{array}
[c]{cl}%
0 & \text{if }f(x)\leq\alpha-1/n\\
\in(0,1) & \text{if }\alpha-1/n<\varphi(x)<\alpha\\
1 & \text{if }f(x)\geq\alpha\text{ (i.e., }x\in K).\text{ }%
\end{array}
\right.  \label{Ury}%
\end{equation}

See \cite{Zhou}, p. 1814. Since $I$ is monotone, the sequence $(I(f_{n}%
^{K}))_{n}$ is also nonincreasing and bounded from below by $0$, which implies
(due to the order continuity of the norm of $E)$, that it is also convergent
in the norm topology of $E.$

This allows us to define $\boldsymbol{\mu}$ on the sets $K\in\Sigma_{up}%
^{+}(X)$ by the formula
\begin{equation}
\boldsymbol{\mu}(K)=\lim_{n\rightarrow\infty}I(f_{n}^{K}). \label{defmu}%
\end{equation}

The definition of $\boldsymbol{\mu}(K)$ is independent of the particular
sequence $(f_{n}^{K})_{n}$ with the aforementioned properties. Indeed, if
$(g_{n}^{K})_{n}$ is another such sequence, fix a positive integer $m$, and
infer from Lemma \ref{lem-aux-2} $(a)$ that $\lim_{n\rightarrow\infty}%
I(g_{n}^{A})\leq I(f_{m}^{A})$. Taking the limit as $m\rightarrow\infty$ on
the right-hand side we get
\[
\lim_{n\rightarrow\infty}I(g_{n}^{A})\leq\lim_{m\rightarrow\infty}I(f_{m}%
^{A}).
\]
Then, interchanging $(f_{n}^{K})_{n}$ and $(g_{n}^{K})_{n}$ we conclude that
actually equality holds.

Clearly, the set function $\mathbf{\mu}:\Sigma_{up}^{+}(X)\rightarrow E_{+}$
is a vector capacity and it takes values in the order interval $[0,I(1)].$

We next show that $\boldsymbol{\mu}$ is upper continuous on $\Sigma_{up}%
^{+}(X)$, that is,%
\[
\boldsymbol{\mu}(K)=\lim_{n\rightarrow\infty}\boldsymbol{\mu}(K_{n})
\]
whenever $(K_{n})_{n}$ is a nonincreasing sequence of sets in $\Sigma_{up}%
^{+}(X)$ such that $K=\cap_{n=1}^{\infty}K_{n}$. Indeed, using formula
(\ref{Ury}) one can choose a nonincreasing sequence of continuous function
$g_{n}$ such that $g_{n}\geq\chi_{K_{n}},$ $g_{n}=0$ outside the neighborhood
of radius $1/n$ of $K_{n}$ and $I(g_{n})-\boldsymbol{\mu}(K_{n})\rightarrow0$.
See formula (\ref{Ury}) and using analogous reasonings with those concerning
relations (7)-(9) in \cite{Zhou}, pp. 1814-1815. Then $\boldsymbol{\mu
}(K)=\lim_{n\rightarrow\infty}I(g_{n}),$ which implies the equality
$\boldsymbol{\mu}(K)=\lim_{n\rightarrow\infty}\boldsymbol{\mu}(K_{n}).$

The next goal is the representation formula (\ref{reprformula}). For this, let
$x^{\ast}\in E_{+}^{\ast}$ be arbitrarily fixed and consider the comonotonic
additive and monotone functional
\[
x^{\ast}\circ I:C(X)\rightarrow\mathbb{R}.
\]
It verifies $I^{\ast}(1)=x^{\ast}(I(1))>0,$ so by Theorem 1 in \cite{Zhou}
there is a unique upper continuous capacity $\nu^{\ast}:\Sigma_{up}%
^{+}(X)\rightarrow\lbrack0,I^{\ast}(1)]$ such that
\[
\left(  x^{\ast}\circ I\right)  (f)=(C)\int_{X}f\mathrm{d}\nu^{\ast}\text{ for
all }f\in C(X).
\]
The capacity $\nu^{\ast}$ is obtained via an approximation process similar to
(\ref{defmu}). Therefore%
\[
\nu^{\ast}(K)=\lim_{n\rightarrow\infty}x^{\ast}(I(f_{n}^{K}))=x^{\ast}%
(\lim_{n\rightarrow\infty}[I(f_{n}^{K}))=x^{\ast}(\boldsymbol{\mu}(K))
\]
for all $K\in\Sigma_{up}^{+}(X),$ which implies
\[
x^{\ast}(I(f))=(C)\int_{X}f\mathrm{d}(x^{\ast}\circ\boldsymbol{\mu}).
\]
Since $x^{\ast}\in E_{+}^{\ast}$ was arbitrarily fixed, an appeal to Lemma
\ref{lem1} $(c)$ easily yields the equality
\[
I(f)=(\operatorname*{CB})\int_{X}f\mathrm{d}\boldsymbol{\mu}.
\]

For the second part of Theorem \ref{thmEWgen},since $\boldsymbol{\mu}$ takes
values in an order bounded interval, one can extend it to all Borel subsets of
$X$ via the formula%
\[
\boldsymbol{\mu}(A)=\sup\left\{  \boldsymbol{\mu}(K):K\text{ closed, }K\subset
A\right\}  ,\text{\quad}A\in\mathcal{B}(X).
\]

The fact that the resulting set function $\boldsymbol{\mu}$ is a vector
capacity is immediate.

This set function $\boldsymbol{\mu}$ also verifies the regularity condition
$(R2)$. Indeed, given a closed set $K$, we can consider the sequence of open
sets%
\[
O_{n}=\left\{  x\in X:d(x,K)<1/n\right\}  .
\]
Clearly, $\boldsymbol{\mu}(K)\leq\boldsymbol{\mu}(O_{n})\leq\boldsymbol{\mu
}(K_{n})$ where
\[
K_{n}=\left\{  x\in X:d(x,K)\leq1/n\right\}  .
\]

Since $\boldsymbol{\mu}$ is upper continuous on closed sets, it follows that
$\lim_{n\rightarrow\infty}\boldsymbol{\mu}(K_{n})=\boldsymbol{\mu}(K)$, whence
$\lim_{n\rightarrow\infty}\boldsymbol{\mu}(O_{n})=\boldsymbol{\mu}(K).$
Therefore $\boldsymbol{\mu}$ verifies the regularity condition $(R2)$. The
uniqueness of the extension of $\boldsymbol{\mu}$ to $\mathcal{B}(X)$ is
motivated by the condition $(R1).$
\end{proof}

\begin{remark}
\label{rem2}If the operator $I$ is submodular, that is,%
\[
I(\sup\left\{  {f,g}\right\}  )+I(\inf\left\{  {f,g}\right\}  )\leq
I(f)+I(g)\text{\quad for all }f,g\in C(X),
\]
then the vector capacity $\boldsymbol{\mu}:\Sigma_{up}^{+}(X)\rightarrow
E_{+}$ stated by Theorem \emph{\ref{thmEWgen}} is submodular. This is a
consequence of Theorem $13$ $(c)$ in \emph{\cite{CMMM2012}}. For the
convenience of the reader we will recall here the argument. Let $A,B\in
\Sigma_{up}^{+}(X)$ and consider the sequences $(f_{n}^{A}(x))_{n}$ and
$(f_{n}^{B}(x))_{n}$ of continuous functions associated respectively to $A$
and $B$ by the formula \emph{(\ref{defmu})}. Then
\[
\boldsymbol{\mu}(A)=\lim_{n\rightarrow\infty}I(f_{n}^{A}),\ \text{
}\boldsymbol{\mu}(B)=\lim_{n\rightarrow\infty}I(f_{n}^{B}),\ \text{
}\boldsymbol{\mu}(A\cup B)=\lim_{n\rightarrow\infty}I(\sup\left\{  f_{n}%
^{A},f_{n}^{B}\right\}  \ \text{ }%
\]
and $\boldsymbol{\mu}(A\cap B)=\lim_{n\rightarrow\infty}I(\inf\left\{
f_{n}^{A},f_{n}^{B}\right\}  $. Since $I$ is submodular, it follows that
\[
\boldsymbol{\mu}(A\cup B)+\boldsymbol{\mu}(A\cap B)\leq\boldsymbol{\mu
}(A)+\boldsymbol{\mu}(B),
\]
and the proof is done.
\end{remark}

\section{The case of operators with bounded variation}

The representation Theorem \ref{thmEWgen} can be extended outside the
framework of monotone operators by considering the class of operators with
bounded variation.

As above, $X$ is a compact Hausdorff space and $E$ is a Banach lattice with
order continuous norm.

\begin{definition}
\label{defOpBV}An operator $I:C(X)\rightarrow E$ has \emph{bounded variation}
\emph{over an order interval} $[f,g]$ if%
\begin{equation}
\vee_{f}^{g}I=\sup%
{\displaystyle\sum\nolimits_{k=0}^{n}}
\left\vert I(f_{k})-I(f_{k-1})\right\vert \text{\quad exists in }E,
\label{sumvar}%
\end{equation}
the supremum being taken over all finite chains $f=f_{0}\leq f_{1}\leq
\cdots\leq f_{n}=g$ of functions in the Banach lattice $C(X).$ The operator
$I$ is said to have \emph{bounded variation} if it has bounded variations on
all\emph{ }order intervals $[f,g]$ in $C(X).$
\end{definition}

Clearly, if $I$ is monotone, then $\vee_{f}^{g}I=I(g)-I(f)$ for all $f\leq g$
in $C(X)$ and thus $I$ has bounded variation.

More generally, every operator $I:C(X)\rightarrow E$ which can be represented
as the difference $I=I_{1}-I_{2}$ of two monotone operators $I_{1}%
,I_{2}:C(X)\rightarrow E$ has bounded variation. This follows from the order
completeness of $E$ and the modulus inequality, which provides an upper bound
for the sums appearing in formula (\ref{sumvar}):
\[%
{\displaystyle\sum\nolimits_{k=0}^{n}}
\left\vert I(f_{k})-I(f_{k-1})\right\vert \leq I_{1}(g)-I_{1}(f)+I_{2}%
(g)-I_{2}(f).
\]
Remarkably, the converse also holds. The basic ingredient is the following result.

\begin{lemma}
\label{lemtech}Suppose that $I:C(X)\rightarrow E$ is a comonotonic additive
operator with bounded variation. Then there exist two positively homogeneous,
translation invariant and monotone operators $I_{1},I_{2}:C(X)\rightarrow E$
such that $I=I_{1}-I_{2}.$

Moreover, if $I$ is upper continuous, then both operators $I_{1}$ and $I_{2}$
can be chosen to be upper continuous.
\end{lemma}

\begin{proof}
The proof is done in the footsteps of Lemma 14 in \cite{CMMM2012} by noticing
first the following four facts:

$(a)$ According to our hypotheses,
\[
I(\alpha f+\beta)=\alpha I(f)+I(\beta)
\]
for all $f\in C(X),$ $\alpha\in\mathbb{R}_{+}$ and $\beta\in\mathbb{R}.$

$(b)$ $\vee_{0}^{f+\alpha}I=\vee_{-\alpha}^{f}I$ for all $f\in C(X)$ and
$\alpha\in\mathbb{R}$ with $f+\alpha\geq0.$

This follows from the definition of the variation.

$(c)$ $\vee_{0}^{\alpha f}I=\alpha\vee_{0}^{f}I$ for all $f\in C(X)_{+}$ and
$\alpha\in\mathbb{R}_{+}.$

Indeed for every $\varepsilon\in E,$ $\varepsilon>0,$ there exists a chain
$0=f_{0}\leq f_{1}\leq\cdots\leq f_{n}=f$ such that
\[%
{\displaystyle\sum\nolimits_{k=0}^{n}}
\left\vert I(f_{k})-I(f_{k-1})\right\vert \geq\vee_{0}^{f}I-\varepsilon.
\]
According to fact $(a)$ the chain $0=\alpha f_{0}\leq\alpha f_{1}\leq
\cdots\leq\alpha f_{n}=\alpha f$ verifies%
\[
\vee_{0}^{\alpha f}I\geq%
{\displaystyle\sum\nolimits_{k=0}^{n}}
\left\vert I(\alpha f_{k})-I(\alpha f_{k-1})\right\vert \geq\alpha\vee_{0}%
^{f}I-\alpha\varepsilon.
\]
As $\varepsilon>0$ was arbitrarily fixed it follows that $\vee_{0}^{\alpha
f}I\geq\alpha\vee_{0}^{f}I.$ By replacing $\alpha$ to $1/\alpha$ and then $f$
by $\alpha f$, one obtains the reverse inequality, $\vee_{0}^{\alpha f}%
I\leq\alpha\vee_{0}^{f}I.$

$(d)$ $\vee_{-\alpha}^{f}I=\vee_{-\alpha}^{0}I+\vee_{0}^{f}I=\vee_{0}^{\alpha
}I+\vee_{0}^{f}I$ for all $f\in C(X)_{+}$ and $\alpha\in\mathbb{R}_{+}.$

The fact that $\vee_{-\alpha}^{f}I\geq\vee_{-\alpha}^{0}I+\vee_{0}^{f}%
I=\vee_{0}^{\alpha}I+\vee_{0}^{f}I$ is a direct consequence of the definition
of variation. For the other inequality, fix arbitrarily $\varepsilon>0$ in $E$
and choose a chain $-\alpha=f_{0}\leq f_{1}\leq\cdots\leq f_{n}=f$ such that
\[%
{\displaystyle\sum\nolimits_{k=0}^{n}}
\left\vert I(f_{k})-I(f_{k-1})\right\vert \geq\vee_{-\alpha}^{f}%
I-\varepsilon.
\]
Then $-\alpha=-f_{0}^{-}\leq-f_{1}^{-}\leq\cdots\leq-f_{n}^{-}=0$ and
$0=f_{0}^{+}\leq f_{1}^{+}\leq\cdots\leq f_{n}^{+}=f.$ Since all pairs
$-f_{k}^{-},f_{k}^{+}$ are comonotonic,

we have%
\begin{align*}
\vee_{-\alpha}^{0}I+\vee_{0}^{f}I  &  \geq%
{\displaystyle\sum\nolimits_{k=0}^{n}}
\left\vert I(-f_{k}^{-})-I(-f_{k-1}^{-})\right\vert +%
{\displaystyle\sum\nolimits_{k=0}^{n}}
\left\vert I(f_{k}^{+})-I(f_{k-1}^{+})\right\vert \\
&  \geq%
{\displaystyle\sum\nolimits_{k=0}^{n}}
\left\vert I(-f_{k}^{-})-I(-f_{k-1}^{-})+I(f_{k}^{+})-I(f_{k-1}^{+}%
)\right\vert \\
&  =%
{\displaystyle\sum\nolimits_{k=0}^{n}}
\left\vert I(f_{k})-I(f_{k-1})\right\vert \geq\vee_{-\alpha}^{f}I-\varepsilon
\end{align*}
and it remains to take the supremum over $\varepsilon>0.$

Now we can proceed to the choice of the operators $I_{1}$ and $I_{2}.$ By
definition,
\[
I_{1}(f)=\vee_{0}^{f}I\text{\quad if }f\in C(X)_{+},
\]
while if $f\in C(X)$ is an arbitrary function, one choose $\alpha\in
\mathbb{R}_{+}$ such that $f+\alpha\geq0$ and put
\[
I_{1}(f)=\vee_{0}^{f+\alpha}I-\alpha\vee_{0}^{1}I.
\]
The fact that $I_{1}$ is well-defined (that is, independent of $\alpha)$ can
be proved as follows. Suppose that $\alpha,\beta>0$ are two numbers such that
$f+\alpha\geq0$ and $f+\beta\geq0.$ Without loss of generality we may assume
that $\alpha<\beta.$ Indeed, according to the facts $(b)-(d)$, we have%
\begin{align*}
\vee_{0}^{f+\beta}I-\beta\vee_{0}^{1}I  &  =\vee_{0}^{f+\alpha+(\beta-\alpha
)}I-\beta\vee_{0}^{1}I\overset{\text{fact }(b)}{=}\vee_{-(\beta-\alpha
)}^{f+\alpha}I-\beta\vee_{0}^{1}I\\
&  \overset{\text{fact }(d)}{=}\vee_{-(\beta-\alpha)}^{0}I+\vee_{0}^{f+\alpha
}I-\beta\vee_{0}^{1}I\\
&  \overset{\text{fact }(b)}{=}\vee_{0}^{\beta-\alpha}I+\vee_{0}^{f+\alpha
}I-\beta\vee_{0}^{1}I\\
&  \overset{\text{fact }(c)}{=}(\beta-\alpha)\vee_{0}^{1}I+\vee_{0}^{f+\alpha
}I-\beta\vee_{0}^{1}I\\
&  =\vee_{0}^{f+\alpha}I-\alpha\vee_{0}^{1}I.
\end{align*}

By definition,
\[
I_{2}=I-I_{1}.
\]

Let $f,g$ be two functions in $C(X)$ such that $f\leq g$ and let $\alpha>0$
such that $f+\alpha\geq0.$ Since $I$ is monotonic and has bounded variation,%
\[
0\leq I(g)-I(f)=I(g+\alpha)-I(f+\alpha)\leq\vee_{f+\alpha}^{g+\alpha}I\leq
\vee_{0}^{g+\alpha}I-\vee_{0}^{f+\alpha}I=I_{1}(g)-I_{1}(f),
\]
whence we infer that both operators $I_{1}$ and $I_{2}$ are monotonic.

Due to the fact $(c),$ the operator $I_{1}$ is positively homogeneous.
Therefore the same is true for $I_{2}.$

For the property of translation invariance, let $f\in C(X)$, $\beta
\in\mathbb{R}$ and choose $\alpha>0$ such that $f+\beta+\alpha\geq0$ and
$\beta+\alpha\geq0.$ Then%
\begin{align*}
I_{1}(f+\beta)  &  =\vee_{0}^{f+\beta+\alpha}I-\alpha\vee_{0}^{1}I\\
&  =I_{1}(f+\beta+\alpha)-\alpha\vee_{0}^{1}I
\end{align*}
while from facts $(b)\&(d)$ we infer that $I_{1}(f)=I_{1}(f+\beta
+\alpha)-\left(  \beta+\alpha\right)  \vee_{0}^{1}I.$ Therefore
\[
I_{1}(f+\beta)=I_{1}(f)+I_{1}(\beta)
\]
which proves that indeed $I_{1}$ is translation invariant. The same holds for
$I_{2}=I-I_{1}.$

As concerns the second part of Lemma \ref{lemtech}, it suffices to prove that
$I_{1}$ is upper continuous when $I$ has this property.

Let $(f_{n})_{n}$ be a decreasing sequence in $C(X)$ which converges pointwise
to a function $f$ also in $C(X).$ By Dini's lemma, $\left\Vert f_{n}%
-f\right\Vert \rightarrow0$. Since $I_{1}$ is a Lipschitz operator (see Lemma
\ref{lem-aux-1*}) we conclude that $\left\Vert I_{1}(f_{n})-I(f)\right\Vert
\rightarrow0.$

This ends the proof of Lemma \ref{lemtech}.
\end{proof}

Now it is clear that the representation Theorem \ref{thmEWgen} can be extended
to the framework of comonotonic additive operators $I:C(X)\rightarrow E$ with
bounded variation by considering Choquet-Bochner integrals associated to
differences of vector capacities (that is, to set functions with bounded
variation in the sense of Aumann and Shapley \cite{AS1974}).

Let $\Sigma$ be a lattice of sets of $X$ and $\boldsymbol{\mu}:\mathcal{\Sigma
}\rightarrow E$ a set function taking values in a Banach lattice $E$ with
order continuous norm.

The \emph{variation} of the set function $\boldsymbol{\mu}$ is the set
function $\left\vert \boldsymbol{\mu}\right\vert $ defined by the formula%
\[
\left\vert \mathbf{\mu}\right\vert (A)=\sup%
{\displaystyle\sum\nolimits_{k=0}^{n}}
\left\vert \boldsymbol{\mu}(A_{k})-\boldsymbol{\mu}(A_{k-1})\right\vert
\text{\quad for }A\in\Sigma,
\]
where the supremum is taken over all finite chains $A_{0}=\emptyset\subset
A_{1}\subset\cdots\subset A_{n}\subset A$ of sets in the lattice $\Sigma.$

The space of set functions with bounded variation,%
\[
\operatorname*{bv}(\Sigma,E)=\left\{  \boldsymbol{\mu}:\Sigma\rightarrow
E:\boldsymbol{\mu}(\emptyset)=0\text{ and }\left\vert \boldsymbol{\mu
}\right\vert (X)\text{ exists in }E\right\}  ,
\]
is a normed vector space when endowed with the norm%
\[
\left\Vert \boldsymbol{\mu}\right\Vert _{\operatorname*{bv}}=\left\Vert
\left\vert \boldsymbol{\mu}\right\vert (X)\right\Vert .
\]

Associated to every set function $\boldsymbol{\mu}\in\operatorname*{bv}%
(\Sigma,E)$ are two positive vector-valued set functions, the\emph{ inner
upper variation} of $\mu,$ defined by%
\[
\boldsymbol{\mu}^{+}(A)=\sup%
{\displaystyle\sum\nolimits_{k=0}^{n}}
\left(  \boldsymbol{\mu}(A_{k})-\boldsymbol{\mu}(A_{k-1})\right)
^{+}\text{\quad for }A\in\Sigma
\]
and the \emph{inner lower variation} of $\boldsymbol{\mu},$ defined by%
\[
\boldsymbol{\mu}^{-}(A)=\sup%
{\displaystyle\sum\nolimits_{k=0}^{n}}
\left(  \boldsymbol{\mu}(A_{k})-\boldsymbol{\mu}(A_{k-1})\right)
^{-}\text{\quad for }A\in\Sigma;
\]
in both cases the supremum is taken over all finite chains $A_{0}%
=\emptyset\subset A_{1}\subset\cdots\subset A_{n}\subset A$ of sets in
$\Sigma.$ Notice that%
\[
\boldsymbol{\mu}=\boldsymbol{\mu}^{+}-\boldsymbol{\mu}^{-}\text{ and
}\left\vert \boldsymbol{\mu}\right\vert =\boldsymbol{\mu}^{+}+\boldsymbol{\mu
}^{-}.
\]

\begin{lemma}
\label{lemmesbv}We assume that $E$ is a Banach lattice with order continuous
norm. The following conditions are equivalent for a set function
$\boldsymbol{\mu}:\Sigma\rightarrow E:$

$(a)$ $\boldsymbol{\mu}\in\operatorname*{bv}(\Sigma,E);$

$(b)$ $\boldsymbol{\mu}^{+}$\ and $\boldsymbol{\mu}^{-}$ are vector
capacities$;$

$(c)$ there exist two vector capacities $\boldsymbol{\mu}_{1}$ and
$\boldsymbol{\mu}_{2}$ on $\Sigma$ such that $\boldsymbol{\mu}=\boldsymbol{\mu
}_{1}-\boldsymbol{\mu}_{2}.$

Moreover, for any such decomposition we have $\boldsymbol{\mu}^{+}%
\leq\boldsymbol{\mu}_{1}$ and $\boldsymbol{\mu}^{-}\leq\boldsymbol{\mu}_{2}$.
\end{lemma}

The details are similar to those presented in \cite{AS1974}, Ch. 1, \S 4.

We are now in a position to state the main result of this section.

\begin{theorem}
\label{thmfinal}Let $X$ be a compact Hausdorff space and $E$ be a Banach
lattice with order continuous norm. Then for every comonotonic additive
operator $I:C(X)\rightarrow E$ with bounded variation there exists a unique
upper continuous set function $\boldsymbol{\mu}:\Sigma_{up}^{+}(X)\rightarrow
E$ with bounded variation such that%
\begin{align}
I(f)  &  =\int_{0}^{+\infty}\boldsymbol{\mu}\left(  \{x:f(x)\geq t\}\right)
\mathrm{d}t\label{genreprformula}\\
&  +\int_{-\infty}^{0}\left[  \boldsymbol{\mu}\left(  \{x:f(x)\geq t\}\right)
-\mu(A)\right]  \mathrm{d}t.\nonumber
\end{align}
for all $f\in C(X).$
\end{theorem}

\begin{proof}
By Lemma \ref{lemtech}, there exist two functionals $I_{1},I_{2}%
:C(X)\rightarrow E$ that are monotone, translation invariant, positively
homogeneous, upper continuous and such that $I=I_{1}-I_{2}$. Define
$\boldsymbol{\mu}:\Sigma_{up}^{+}(X)\rightarrow E$ by $\boldsymbol{\mu
}=\boldsymbol{\mu}_{1}-\boldsymbol{\mu}_{2},$ where $\boldsymbol{\mu}%
_{1},\boldsymbol{\mu}_{2}$ are associated to the functionals $I_{1}$ and
$I_{2}$ via Theorem \ref{thmEWgen}. Then $\boldsymbol{\mu}$ is upper
continuous and by Lemma \ref{lemmesbv} it also has bounded variation. We now
prove that the representation formula (\ref{genreprformula}) holds.

Suppose that $f\in C(X)_{+}.$ The integral%
\[
\int_{0}^{+\infty}\boldsymbol{\mu}\left(  \{x:f(x)\geq t\}\right)
\mathrm{d}t
\]
is well defined as a Bochner integral (see Lemma \ref{lem6}). Given
$\varepsilon>0,$ one can choose an equidistant division $0=t_{0}<\cdots
<t_{m}=\left\Vert f\right\Vert $ of $[0,\left\Vert f\right\Vert ]$ such that%
\begin{align}
&  \left\Vert \int_{0}^{+\infty}\boldsymbol{\mu}\left(  \{x:f(x)\geq
t\}\right)  \mathrm{d}t-\sum_{k=0}^{m-1}\boldsymbol{\mu}\left(  \{x:f(x)\geq
t_{k}\}\right)  (t_{k+1}-t_{k})\right\Vert \nonumber\\
&  =\left\Vert \int_{0}^{\left\Vert f\right\Vert }\boldsymbol{\mu}\left(
\{x:f(x)\geq t\}\right)  \mathrm{d}t-\sum_{k=0}^{m-1}\boldsymbol{\mu}\left(
\{x:f(x)\geq t_{k}\}\right)  (t_{k+1}-t_{k})\right\Vert <\varepsilon
\label{cond0}%
\end{align}
and $\left\Vert f\right\Vert /m<\varepsilon$.

Denote $C_{k}=\{x:f(x)\geq t_{k}\}$ for $k=0,...,m-1.$ By the definition of
$\boldsymbol{\mu}_{1}$ and $\boldsymbol{\mu}_{2}$ (see the proof of Theorem
\ref{thmEWgen}) one can choose functions $f_{n}^{C_{k}}\in C(X)_{+}$ such that%
\begin{equation}
\left\Vert \boldsymbol{\mu}\left(  C_{k}\right)  -I(f_{n}^{C_{k}})\right\Vert
<\varepsilon/\left\Vert f\right\Vert \text{ and }n^{-1}<\left\Vert
f\right\Vert /m. \label{cond1}%
\end{equation}

Because the functions $f_{n}^{C_{k}}$ are defined by the formula (\ref{Ury})
and $n^{-1}<\left\Vert f\right\Vert /m$, it follows that the functions
$f_{n}^{C_{i}}(t_{i+1}-t_{i})$ and $\sum_{k=i+1}^{m-1}f_{n}^{C_{k}}%
(t_{k+1}-t_{k})$ are comonotonic for all indices $i,$ so that%
\begin{align*}
I(\sum_{k=0}^{m-1}f_{n}^{C_{k}}(t_{k+1}-t_{k}))  &  =I(f_{n}^{C_{0}}%
)(t_{1}-t_{0})+I(\sum_{k=1}^{m-1}f_{n}^{C_{k}}(t_{k+1}-t_{k}))\\
&  =\cdots=\sum_{k=0}^{m-1}I(f_{n}^{C_{k}})(t_{k+1}-t_{k});
\end{align*}
the property of positive homogeneity of $I$ is assured by Lemma \ref{lemtech}.

Notice that%
\[
f(x)\leq\sum_{k=0}^{m-1}f_{n}^{C_{k}}(x)(t_{k+1}-t_{k})\leq f(x)+2\varepsilon
\text{\quad for all }x\in X.
\]
Since the operators $I_{1},I_{2}$ are monotone and translation invariant, it
follows that%
\[
I_{j}(f)\leq I_{j}(\sum_{k=0}^{m-1}f_{n}^{C_{k}}(t_{k+1}-t_{k}))\leq
I_{j}(f)+2\varepsilon I_{j}(1)\text{\quad for }j\in\{1,2\}.
\]
whence%
\begin{equation}
\left\Vert I_{j}(\sum_{k=0}^{m-1}f_{n}^{C_{k}}(t_{k+1}-t_{k}))-I_{j}%
(f)\right\Vert \leq2\varepsilon\left\Vert I_{j}(1)\right\Vert \text{\quad for
}j\in\{1,2\}. \label{cond2}%
\end{equation}

Therefore%

\begin{align*}
&  \left\Vert \int_{0}^{+\infty}\boldsymbol{\mu}\left(  \{x:f(x)\geq
t\}\right)  \mathrm{d}t-I(f)\right\Vert \\
&  \leq\left\Vert \int_{0}^{+\infty}\boldsymbol{\mu}\left(  \{x:f(x)\geq
t\}\right)  \mathrm{d}t-\sum_{k=0}^{m-1}I(f_{n}^{C_{k}})(t_{k+1}%
-t_{k})\right\Vert \\
&  +\left\Vert \sum_{k=0}^{m-1}I(f_{n}^{C_{k}})(t_{k+1}-t_{k})-I(f)\right\Vert
\\
&  \overset{\text{see }(\ref{cond2})}{\leq}\left\Vert \int_{0}^{+\infty
}\boldsymbol{\mu}\left(  \{x:f(x)\geq t\}\right)  \mathrm{d}t-\sum_{k=0}%
^{m-1}I(f_{n}^{C_{k}})(t_{k+1}-t_{k})\right\Vert \\
&  +2\varepsilon\left\Vert I_{1}(1)\right\Vert +2\varepsilon\left\Vert
I_{2}(1)\right\Vert \\
&  \leq\left\Vert \int_{0}^{+\infty}\boldsymbol{\mu}\left(  \{x:f(x)\geq
t\}\right)  \mathrm{d}t-\sum_{k=0}^{m-1}\boldsymbol{\mu}\left(  C_{k}\right)
(t_{k+1}-t_{k})\right\Vert \\
&  +\left\Vert \sum_{k=0}^{n-1}\boldsymbol{\mu}\left(  C_{k}\right)
(t_{k+1}-t_{k})-\sum_{k=0}^{n-1}I(f_{n}^{C_{k}})(t_{k+1}-t_{k})\right\Vert \\
&  +2\varepsilon\left\Vert I_{1}(1)\right\Vert +2\varepsilon\left\Vert
I_{2}(1)\right\Vert \\
&  \overset{\text{see }(\ref{cond0})\&(\ref{cond1})}{\leq}2\varepsilon\left(
1+\left\Vert I_{1}(1)\right\Vert +\left\Vert I_{2}(1)\right\Vert \right)  .
\end{align*}
Since $\varepsilon>0$ was arbitrarily fixed, the above reasoning yields
formula (\ref{genreprformula}) in the case of positive functions.

If $f\notin C(X)_{+},$ then $f+\left\Vert f\right\Vert \in C(X)_{+~}$ and by
the preceding considerations we have%
\begin{multline*}
I(f)+\left\Vert f\right\Vert I(1)=I(f+\left\Vert f\right\Vert )=\int
_{0}^{+\infty}\boldsymbol{\mu}\left(  \{x:f(x)+\left\Vert f\right\Vert \geq
t\}\right)  \mathrm{d}t\\
=\int_{0}^{+\infty}\boldsymbol{\mu}\left(  \{x:f(x)\geq t\}\right)
\mathrm{d}t+\int_{-\left\Vert f\right\Vert }^{0}\boldsymbol{\mu}\left(
\{x:f(x)\geq t\}\right)  \mathrm{d}t\\
=\int_{0}^{+\infty}\boldsymbol{\mu}\left(  \{x:f(x)\geq t\}\right)
\mathrm{d}t+\int_{-\left\Vert f\right\Vert }^{0}\boldsymbol{\mu}\left(
\{x:f(x)\geq t\}\right)  -\boldsymbol{\mu}(X)\mathrm{d}t+\left\Vert
f\right\Vert \boldsymbol{\mu}(X)\\
=\int_{0}^{+\infty}\boldsymbol{\mu}\left(  \{x:f(x)\geq t\}\right)
\mathrm{d}t+\int_{-\infty}^{0}\boldsymbol{\mu}\left(  \{x:f(x)\geq t\}\right)
-\boldsymbol{\mu}(X)\mathrm{d}t+\left\Vert f\right\Vert I(1).
\end{multline*}
The proof of the representation formula (\ref{genreprformula}) is now complete.

As concerns the uniqueness of $\boldsymbol{\mu}$, suppose that
$\boldsymbol{\nu}$ is another upper continuous monotone set function with
bounded variations for which the formula (\ref{genreprformula}) holds. Given a
set $K=\left\{  x:f(x)\geq t\right\}  \in\Sigma_{up}^{+}(X),$ it is known that
the functions $f_{n}^{K}:X\rightarrow\lbrack0,1]$ defined by the formula
(\ref{Ury}) decrease to$\chi_{K}.$ This implies that $(\left\{  x:f_{n}%
^{K}(x)\geq t\right\}  )_{n}$ is decreasing to $K$ for each $t\in\lbrack0,1].$
Consider the sequence of functions
\[
\varphi_{n}:[0,1]\rightarrow E,\text{\quad}\varphi_{n}(t)=\boldsymbol{\nu
}(\left\{  x:f_{n}^{K}(x)\geq t\right\}  ).
\]
Notice that all these functions have bounded variation and their variation is
bounded by the variation of $|\boldsymbol{\nu}|(X).$

By Lebesgue dominated convergence,
\[
\lim_{n\rightarrow\infty}\int_{0}^{1}\varphi_{n}(t)dt=\boldsymbol{\nu}(K).
\]
On other hand, by (\ref{genreprformula}) and the definition of
$\boldsymbol{\mu},$%
\[
\boldsymbol{\mu}(K)=\lim_{n\rightarrow\infty}I(f_{n}^{K})=\lim_{n\rightarrow
\infty}\int_{0}^{1}\varphi_{n}(t)dt=\boldsymbol{\nu}(K),
\]
which ends the proof of the uniqueness.
\end{proof}

\section{Open problems}

We end our paper by mentioning few open problems that might be of interest to
our readers.

\begin{problem}
\label{probl1}Is the order continuity of the norm of $E$ a necessary condition
for the validity of Theorems \ref{thmEWgen} and \ref{thmfinal}?
\end{problem}

As was noticed by Bartle, Dunford and Schwartz \cite{BDS}, much of the theory
of weakly compact linear operators defined on a space $C(X)$ is dominated by
the concept of absolute continuity. For more recent contributions see
\cite{N1975}, \cite{N1979} and \cite{N2009}.

Suppose that $\mathcal{A}$ is a $\sigma$-algebra and $E$ is a Banach lattice
with order continuous norm.

A vector capacity $\boldsymbol{\mu}:\mathcal{A}\rightarrow E_{+}$ is called
absolutely continuous with respect to a capacity $\lambda:\mathcal{A}%
\rightarrow\lbrack0,\infty)$ (denoted $\mu\ll\lambda)$ if for every
$\varepsilon>0$ there is $\delta>0$ such that%
\[
A\in\mathcal{A},\text{ }\lambda(A)<\delta\Longrightarrow\left\Vert
\boldsymbol{\mu}(A)\right\Vert <\varepsilon.
\]
The following lemma extends a result proved by Pettis in the context of
$\sigma$-additive measures.

\begin{lemma}
\label{probl2}If $\mu$ is upper continuous and $\lambda$ is upper continuous
and supermodular then the condition $\mu\ll\lambda$ is equivalent to the
following one:
\[
A\in\mathcal{A}\text{ and }\lambda(A)=0\Longrightarrow\boldsymbol{\mu}(A)=0.
\]

\end{lemma}

The proof is immediate, by reductio ad absurdum.

\begin{problem}
\label{probl3}Suppose $\boldsymbol{\mu}:\mathcal{A}\rightarrow E_{+}$ is an
upper continuous vector measure. Does there exist a capacity $\lambda
:\mathcal{A}\rightarrow\lbrack0,\infty)$ such that $\mu\ll\lambda?$

If \ Yes, is it possible to choose $\lambda$ of the form $\lambda=x^{\ast
}\circ\boldsymbol{\mu}$ for a suitable $x^{\ast}\in E_{+}^{\ast}?$
\end{problem}

It would be also interesting the following operator analogue of Problem
\ref{probl3}:

\begin{problem}
Does there exist for each Choquet operator $I:C(X)\rightarrow E$ a functional
$x^{\ast}\in E_{+}^{\ast}$ such that for every $\varepsilon>0$ there is
$\delta>0$ with the property%
\[
\left\Vert I(f)\right\Vert \leq\varepsilon\left\Vert f\right\Vert +\delta
x^{\ast}\left(  \left\vert f\right\vert \right)  \text{\quad for all }f\in
C(X)?
\]

\end{problem}

\textbf{Declaration of interests : none.}

\end{document}